\documentclass{gtart}
\usepackage[utf8]{inputenc}
\usepackage{latexsym}
\usepackage{amsfonts}
\usepackage{amsthm}
\usepackage{amsmath}
\usepackage{graphicx}

\newtheorem{theorem}{Theorem}[section]
\newtheorem{conjecture}[theorem]{Conjecture}
\newtheorem{lemma}[theorem]{Lemma}
\newtheorem{definition}[theorem]{Definition}
\newtheorem{props}[theorem]{Properties}
\theoremstyle{remark}

\newtheorem{example}[theorem]{Example}

\def\D{\partial}
\def\csch{{\rm csch}}
\def\VK{{V^{[\leq{}K]}}}
\def\TnK{{T_n^{[\leq{}K]}}}
\def\glie{{\mathfrak{g}}}
\def\End{{\rm End}}
\def\b{{\bar{\phantom{w}}}}
\def\Maps{{\rm Maps}}
 \def\Aut{{\rm Aut}}

\def\Vr{{V^{[r]}}}
\def\Vbr{{\overline{V}^{[r]}}}

\def\id{{\rm id}}
\def\ad{{\rm ad}}

\def\Z{\mathbb{Z}}
\def\R{\mathbb{R}}
\def\N{\mathbb{N}}
\def\BCH{{\rm BCH}}

\def\H{\mathbb{H}}

\begin{document}

\title{Universal averages in gauge actions}

\authors{Ruth Lawrence and Maor Siboni}

\address{Einstein Institute of Mathematics, Hebrew University of Jerusalem, ISRAEL\\
Racah Institute of Physics, Hebrew University of Jerusalem, ISRAEL}
\email{ruthel.naimark@mail.huji.ac.il}

\begin{abstract}
We give a construction of a universal average of Lie algebra elements whose exponentiation
gives (when there is an associated Lie group) a totally symmetric geometric mean of Lie group elements (sufficiently close to the identity) with the property that in an action of the group on a space $X$ for which $n$ elements all take a particular point $a\in{}X$ to a common point $b\in{}X$, also the mean
will take $a$ to $b$. The construction holds without the necessity for the existence of a Lie group and the universal average $\mu_n(x_1,\ldots,x_n)$ is a totally symmetric
universal expression in the free Lie algebra generated by $x_1,\ldots,x_n$. Its expansion
up to three brackets is found explicitly and various properties of iterated averages are
given. There are applications to the construction of explicit symmetric differential
graded Lie algebra models. This work is based on the second author's minor thesis \cite{Si}.
 \end{abstract}
 \maketitlepage

\section{Introduction}

Suppose that $G$ is a Lie group with corresponding Lie algebra $\glie$. The exponential map
provides a smooth diffeomorphism between a suitable convex centrally symmetric neighbourhood
(say $U$) of $0\in\glie$ and its image, a neighbourhood (which we denote $V$) of $1\in{}G$.
The map  $x\longmapsto-x$ on $U$ induces the map $g\longmapsto{}g^{-1}$  on $V$. The map
$g\longmapsto\frac12g$ on $U$ induces a map on $V$ which we will call square-root and for which
\[\sqrt{g^{-1}}=\left(\sqrt{g}\right)^{-1},\quad(\sqrt{g})^2=g\]

Next observe that for $g,h\in{}G$ sufficiently close to each other so that
 $g^{-1}h\in{}V$ (and thus also $h^{-1}g\in{}V$),
\[g\cdot\sqrt{g^{-1}h}=h\cdot\sqrt{h^{-1}g}\]
which we denote by $k$. Being symmetric in $g$, $h$, this expression can be used to define the
geometric mean of $g$ and $h$, which is neither $\sqrt{gh}$ nor $\sqrt{hg}$ (even when these are
defined, being in general distinct from each other). Apart from its symmetry, the characterising
feature of this expression is that in any action of $G$ on a space $X$ for which $g(a)=h(a)=b$,
some $a,b\in{}X$, also $k(a)=b$.
\[\includegraphics[width=.3\textwidth]{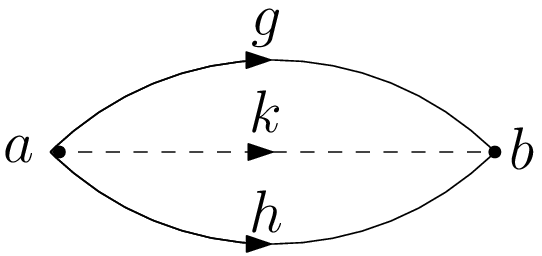}\]

The aim of this paper is to extend this idea of the geometric mean to an arbitrary number of
elements, and in particular, in terms of their logarithms, consider it as a universal average
$\mu_n$ at the Lie algebra level which we show exists as a totally symmetric element of the free Lie algebra on $n$ generators, along with giving explicit formulae for the first few orders and general properties.

In \S2, the Baker-Campbell-Hausdorff formula and its basic properties are recalled. In \S3, we
translate the problem into a purely algebraic one at the Lie algebra level and, using \S2, give a  closed formula for $\mu_2(x,y)$ in the free Lie algebra on two generators. In \S4, we give an
algorithm for $\mu_n$, as the limit of an iterative procedure in terms of $\mu_{n-1}$, including
a proof of its existence and an alternative algebraic characterisation. In \S5, we use this
characterisation to compute the expansion of $\mu_n$ up to third order in brackets while in
\S6, a number of general properties of the $\mu_n$ are discussed. In \S7 we give an example on
$SL(2,\R)$. We conclude in \S8 by indicating an application of universal averages to the
explicit construction of symmetric DGLA models of simple cells.

\section{The Baker-Campbell-Hausdorff formula}

For non-commuting indeterminates $x$ and $y$, there are unique homogeneous (non-commuting)
polynomials $F_n(x,y)$ of degree $n$, for $n\in\mathbb{N}$, such that, as formal series
\[\exp(x).\exp(y)=\exp\left(\sum_{n=1}^\infty{}F_n(x,y)\right)\>.
\]
In particular, $F_1(x,y)=x+y$ and it is a classical result that for $n>1$, $F_n(x,y)$ lies
in the free Lie algebra on the two generators $x$, $y$, that is, it can be expressed as a
linear combination of iterated brackets of $x$, $y$; see \cite{E} for a short proof. The
formula for $\sum_{n=1}^\infty{}F_n(x,y)$ is known as the {\sl Baker-Campbell-Hausdorff formula},
and we will denote it by $\BCH(x,y)$; see \cite{D} for a computational formula.

\begin{props}\label{bchprops}
 \begin{itemize}
 \item[(a)] The first few terms of $\BCH(x,y)$ are
     \begin{align*}
    \BCH(x,y)=x+y&+\frac{1}{2}[x,y]+\frac{1}{12}(X^2y+Y^2x)-\frac{1}{24}XYXy\\
    -\frac{1}{720}(X^4y+Y^4x)&+\frac{1}{120}(X^2Y^2x+Y^2X^2y)
    +\frac{1}{360}(XY^3x+YX^3y)+\cdots
     \end{align*}
     where $X,Y$ denote $\ad_x,\ad_y$.
 \item[(b)] The formula is universal and thus also applies to the operators $\ad_x$, $\ad_y$ for $x,y\in{}A$, in any Lie algebra $A$. By the Jacobi identity, $\BCH(\ad_x,\ad_y)=\ad_{\BCH(x,y)}$ and so in $\Aut(A)$,
     $(\exp{\ad_x})\circ(\exp{\ad_y})=\exp{\ad_{\BCH(x,y)}}$.
 \item[(c)] Uniqueness implies associativity of BCH, that is
  $\BCH\big(\BCH(x,y),z\big)=\BCH\big(x,\BCH(y,z)\big)$ for any symbols $x,y,z$.
  Denote the combined BCH of $n$ symbols $x_1,\ldots{}x_n\in{}A$ by $\BCH(x_1,\ldots,x_n)$ so that
   \[\exp{\BCH(x_1,\ldots,x_n)}=(\exp{x_1})\cdots(\exp{x_n})\>,\]
     in the (completed) universal enveloping algebra of $A$ and \[\exp{\BCH(\ad_{x_1},\ldots,\ad_{x_n})}
     =(\exp{\ad_{x_1}})\cdots(\exp{\ad_{x_n}})\in\Aut(A)\>.
     \]
     Again $\BCH(x_1,\ldots,x_n)$ will be a formal sum of terms, the zeroth order being $x_1+\cdots+x_n$ and higher orders being linear combinations of (repeated) Lie brackets of the $x_i$'s.
 \item[(d)] Uniqueness similarly implies that $\BCH(x,-x)=0$ while \[\BCH(-x_1,\ldots,-x_n)=-\BCH(x_n,\ldots,x_1)\>.\]
\item[(e)] $\BCH(x,y,-x)=(\exp{\ad_x})y$.
\item[(f)] $\BCH(\exp(\ad_e)x,\exp(\ad_e)y))=\exp(\ad_e)\BCH(x,y)$.
\end{itemize}
\end{props}

\section{Algebraic formulation of the problem and $\mu_2$}

We make a non-standard definition of Lie algebra action. Let $\glie$ be a Lie algebra.
\begin{definition}
By a $\glie$-action we will mean a set $X$ along with bijections $u_e:X\longrightarrow{}X$ for any $e\in{}\glie$ such that
\begin{itemize}
\item[(i)] $u_f(u_e(a))=u_{\BCH(e,f)}(a)$ for all $e,f\in\glie$ and $a\in{}X$;
\item[(ii)] $u_0=\id$;
\item[(iii)] if $u_e(a)=a$ for some $e\in\glie$, $a\in{}X$ then $u_{\frac12e}(a)=a$.
\end{itemize}
\end{definition}

Other consequences follow from the properties of $\BCH$. For example, $u_{-e}=(u_e)^{-1}$.

The main theorem of the paper is the following.

\begin{theorem} There exists a totally symmetric expression $\mu_n(x_1,\ldots,x_n)$, in the free Lie algebra on the $n$ generators, $x_1,\ldots,x_n$ with the following property. For any Lie algebra $\glie$ and $\glie$-action on $X$, if $x_1,\ldots,x_n\in\glie$, $a,b\in{}X$ satisfy $u_{x_i}(a)=b$ for all $i=1,\ldots,n$ then \[u_{\mu_n(x_1,\ldots,x_n)}(a)=b\>.\]
\end{theorem}

It will be proved constructively at the end of this section for $n=2$ and in the next section for general $n$.

\begin{example}
Suppose that $G$ is a Lie group acting on a manifold $X$ and $\glie$ is its Lie algebra. Putting  $u_e(a)=(\exp(-e))(a)$ will define a $\glie$-action so long as condition (iii) is satisfied, for example if the action of $G$ is free, although in this case the theorem would be vacuous.
\end{example}

\begin{example}
Suppose that $E$ is a trivialised vector bundle over a base space $B$ with fibre $V$. Let
$\Omega$ be the space of connections on $E$, that is $\End\,V$-valued 1-forms on $B$,
\[\omega=\sum_i\omega_idx_i\]
in local coordinates $(x_i)$ on $B$, with smooth functions
$\omega_i(x)\in\End\,V$. Then $\omega\in\Omega$ defines a notion of flat sections of $E$ with respect to $\omega$, by functions $v:B\longrightarrow{}V$ for which $(d+\omega)v=0$, that is $\frac{\partial{}v}{\partial{}x_i}+\omega_i{}v=0$.

\noindent Let $G$ be the gauge group, consisting of smooth maps
$g:B\longrightarrow{}GL(V)$. It acts on sections by taking $v(x)$ to $g(x)\cdot{}v(x)$. Correspondingly, $G$ acts on $\Omega$ by
\[
g\cdot\omega=-dg\cdot{}g^{-1}+g\omega{}g^{-1}
\]
or in coordinates, by $(g\cdot\omega)_i=-\frac{\partial{}g}{\partial{}x_i}g^{-1}+g\omega_i{}g^{-1}$.
Define $X\subset\Omega$ to be the space of flat connections, solutions of the Maurer-Cartan equation
\[
X=\Big\{\omega\in\Omega\,\Big|\,d\omega+\frac12[\omega,\omega]=0\Big\}
\]
that is, $\omega_i(x)$ such that $\frac{\partial\omega_j}{dx_i}-\frac{\partial\omega_i}{dx_j}+[\omega_i,\omega_j]=0$. The action of $G$ preserves $X$. The infinitesimal gauge group action is that of $\glie=\Maps(X,\End\,V)$ while $e\in\glie$ defines a vector field on $\Omega$ (and on $X$) by
\[e\cdot\omega=-de+[e,\omega]\]
Define $u_e:X\longrightarrow{}X$ to be the flow of the infinitesimal action of $-e$ in unit time, that is, it is the map $\omega(0)\longmapsto\omega(1)$ for solutions of the equation $\dot\omega=de-[e,\omega]$. This can be written
\[a\longmapsto{}e^{-\ad_e}a+\frac{1-e^{-\ad_e}}{\ad_e}de
=(\exp(-e))\omega(\exp{e})-d(\exp(-e))\cdot\exp{e}\]
and is the same as the action of $\exp(-e)$. Note that the term is to be interpreted as a power series in $-\ad_e$. This defines a $\glie$-action on $X$, because $u_e(a)=a$ precisely when $de=\ad_ea$ which is a linear condition on $e$.
\end{example}

\begin{example}
For a regular cell complex $X$, it is possible to associate a DGLA model $A=A(X)$ over $\mathbb{Q}$ satisfying the following conditions
 \begin{itemize}
 \item[(i)] as a Lie algebra, $A(X)$ is freely generated by a set of generators, one for each cell in $X$ and whose grading is one less than the geometric degree of the cell;
 \item[(ii)] vertices (that is $0$-cells) in $X$ give rise to generators $a$ which satisfy the Maurer-Cartan equation $\partial a+\frac12[a,a]=0$ (a flatness condition);
 \item[(iii)] for a cell $x$ in $X$, the part of $\D{x}$ without Lie brackets is the geometric boundary $\D_0x$ (where an orientation must be fixed on each cell);
 \item[(iv)] (locality) for a cell $x$ in $X$, $\D{x}$ lies in the Lie algebra generated by the generators of $A(X)$ associated with cells of the closure $\bar{x}$.
 \end{itemize}

The existence and general construction of such a model was demonstrated by Sullivan in the
  appendix to \cite{TZ}.  By \cite{B4}, there exist consistent (even symmetric) towers of
   models of simplices, and such towers are unique up to (exact) DGLA isomorphism.
   The model of
   an interval is unique \cite{LS}. In \cite{GGL}, an explicit symmetric model of the bi-gon
    (exhibiting the dihedral symmetry of the bi-gon) was given, the main intermediate step
     being the construction of a `symmetric point' in the model of the boundary of the bi-gon,
     invariant under the full symmetries of the bi-gon. Similarly in \cite{GL}, an explicit construction of a model of a single triangle which is invariant under the action of the
     symmetry group $S_3$ of the triangle is given, and again the main intermediate step is the construction of a totally symmetric central `point' (solution of Maurer-Cartan).

 While the inspiration for the construction of such models came from rational homotopy theory
  (\cite{Q}, \cite{S}), their application may be to diverse fields where such infinity
  structures enter, from deformation theory to discretisation of differential equations,
  to be discussed in future work.

  The motivation for the constructions of this paper is precisely this example, in the case of the bi-gon for $\mu_2$ and an appropriate three-dimensional cell for $\mu_n$, $n>2$. For a banana-shaped cell, with two points $a$ and $b$ between which there are $n$ edges and a single three-dimensional cell, the central point of the cell is found at the midpoint of the `diagonal' from $a$ to $b$ given by the mean $\mu_n(x_1,\ldots,x_n)$.
\end{example}

\begin{proof} (Thm 3.2 for $n=2$.) Set $\mu_2(x,y)=\BCH(x,\frac12\BCH(-x,y))$. There are two requirements to check in order to verify that this is a solution.
\begin{itemize}
\item[(a)] If $u_x(a)=u_y(a)=b$ then $u_{\BCH(-x,y)}(b)=u_y(u_{-x}(b))=u_y(a)=b$. Thus $u_{\frac12\BCH(-x,y)}(b)=b$ and hence combining with $u_x(a)=b$ we get $u_{\mu_2(x,y)}(a)=b$.
\item[(b)] Interchanging $x$ and $y$ in the formula for $\mu_2(x,y)$, we get
    \[\BCH(y,\tfrac12\BCH(-y,x))=\BCH(y,\BCH(-y,x),-\tfrac12\BCH(-y,x))\]
     Since $-\BCH(-y,x)=\BCH(-x,y)$, this simplifies to $\BCH(x,\frac12\BCH(x,-y))=\mu_2(x,y)$ as required.
\end{itemize}
\end{proof}

\noindent Using the first few terms in the expansion of $\BCH$, we get
\[\mu_2(x,y)=\frac12(x+y)-\frac1{48}[x,[x,y]]-\frac1{48}[y,[y,x]]+\cdots\]
up to the second order in Lie brackets.

\noindent{\bf Remark:} The formula for $\mu_2(x,y)$ first appeared in \cite{GGL}. It is unique satisfying the conditions of the theorem, as follows from \cite{GGL}, since in the example of the Lie algebra and action coming from the bi-gon, there is only a one-parameter family of flows from $a$ to $b$, namely $\BCH(x,t\BCH(-x,y))$
and only the value $t=\frac12$ (which gives $\mu_2$) is symmetric in $x$ and $y$.

\begin{lemma}
\begin{itemize}
\item[(i)] $\mu_2(-x,-y)=-\mu_2(x,y)$
\item[(ii)] $\mu_2(\BCH(z,x),\BCH(z,y))=\BCH(z,\mu_2(x,y))$
\item[(iii)] $\mu_2(\BCH(x,z),\BCH(y,z))=\BCH(\mu_2(x,y),z)$
\end{itemize}
\end{lemma}
\begin{proof}
\begin{itemize}
\item[(i)] By definition and 2.1(d),
\[\mu_2(-x,-y)=\BCH(-x,\tfrac12\BCH(x,-y))=-\BCH(\tfrac12\BCH(y,-x),x)
\]
But by 2.1(c),(d),(e), $\BCH(y,-x)=\BCH(x,-x,y,-x)=e^X\BCH(-x,y)$ while $e^Xx=x$ where $X=\ad_x$. Thus by 2.1(f)
\[
\mu_2(-x,-y)=-\BCH(e^X\tfrac12\BCH(-x,y),e^Xx)
=-e^X\BCH(\tfrac12\BCH(-x,y),x)
\]
which by 2.1(e) simplifies to $-\BCH(x,\frac12\BCH(-x,y))=-\mu_2(x,y)$.

\item[(ii)]
This follows immediately from the definition, since
\[\BCH(-\BCH(z,x),\BCH(z,y))=\BCH(-x,y)\]

\item[(iii)]
Follows by combining (i), (ii).
\end{itemize}
\end{proof}

\section{An algorithm for $\mu_n$}

Knowing how to average (symmetrically) pairs of objects using $\mu_2$ from \S3, one can average triples of objects by iteratively averaging pairs. Thus, starting with $x,y,z$, define three sequences $(x_k)$, $(y_k)$, $(z_k)$ in the free Lie algebra $L[x,y,z]$ on three generators $x,y,z$ by
\[
x_{k+1}=\mu_2(y_k,z_k),\quad{}y_{k+1}=\mu_2(x_k,z_k),\quad{}z_{k+1}=\mu_2(x_k,y_k)
\]
with initial conditions $x_0=x$, $y_0=y$, $z_0=z$. Pictorially, representing $x_k$, $y_k$, $z_k$ by points (though they would be edges in the representation of Example 3.5) and $\mu_2(x,y)$ by the midpoint of a line drawn between $x$ and $y$, we get
\[\includegraphics[width=.3\textwidth]{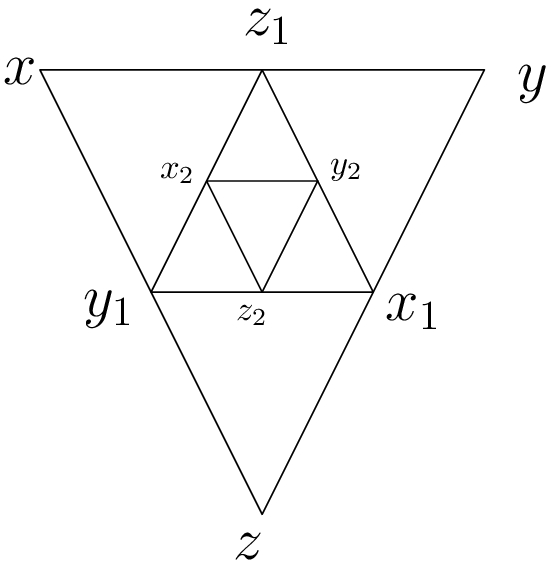}\]
Below we will prove that the three sequences do indeed all converge to a common element of $L[x,y,z]$, which we denote $\mu_3(x,y,z)$. Convergence here means the convergence of the truncated expressions with $\leq{}K$ Lie brackets, for all $K\in\N$. The case of $K=0$ is precisely the geometric picture above with convergence to the centroid.

In the same manner, one can inductively construct each universal average from the previous one, defining $\mu_n(x_1,\ldots,x_n)\in{}L[x_1,\ldots,x_n]$,  in terms of $\mu_{n-1}$.

\begin{lemma}
There is a unique sequence of totally symmetric elements $\mu_n$, $n>2$ in the free Lie algebra $L[x_1,\ldots,x_n]$ such that for each $n>2$, the $n$ sequences $(x_1^k),\ldots,(x_n^k)$ defined iteratively by
\[
x_i^{k+1}=\mu_{n-1}(x_1^k,\ldots,\widehat{x_i^k},\ldots,x_n^k)\eqno{(1)}
\]
with initial conditions $x_i^0=x_i$, $1\leq{}i\leq{}n$, will all converge to $\mu_n$ as
$k\longrightarrow\infty$, in the sense that their truncations with $\leq{}K$ Lie brackets
will converge, for all $K\in\N$. The part of $\mu_n$ without Lie brackets is
$\frac1n(x_1+\cdots+x_n)$.
\end{lemma}
\begin{proof}
The proof is by induction on $n$, starting with $n=3$.
As it stands, the relations (1) are highly non-linear, since $\mu_{n-1}$ is non-linear. It is a
notoriously difficult problem to iterate non-linear operations. However, since
$(x_1^k,\ldots,x_n^k)$  are obtained from $(x_1^0,\ldots,x_n^0)$ by iterating $k$ times the
fixed relations (1), it can also be seen that they will be obtained from $(x_1^1,\ldots,x_n^1)$
by $(k-1)$  applications of the same iteration procedure. This means that if in the formulae
for $x_1^{k-1},\ldots,x_n^{k-1}$  we replace $x_1,\ldots,x_n$   by  $x_1^1,\ldots,x_n^1$,
respectively, then we will obtain $x_1^k,\ldots,x_n^k$.
Let $T_n$  denote the operation on the free Lie algebra $L[x_1,\ldots,x_n]$, specified by
substituting $x_1^1,\ldots,x_n^1$  for  $x_1,\ldots,x_n$, respectively, that is,
\[
x_i\rightarrow\mu_{n-1}(x_1,\ldots,\widehat{x_i},\ldots,x_n),\quad1\leq{}i\leq{}n
\]
This is a linear map and by the above argument,  $x_i^{k}=T_n(x_i^{k-1})$  and thus
$x_i^k=(T_n)^k(x_i)$.

For $K\geq0$, let $\VK$ denote the vector subspace of $V=L[x_1,\ldots,x_n]$ spanned by Lie
monomials with $\leq{}K$ brackets. Since $L[x_1,\ldots,x_n]$ is free and the Jacobi relation
preserves the number of brackets, there is a well-defined projection map $V\longrightarrow\VK$
annihilating all Lie monomials with $>K$ brackets. Let $\TnK$ denote the truncation of $T_n$
to $\VK$, so that the truncation of $x_i^k$ to $\VK$ is $(\TnK)^k(x_i)$. To understand what
happens in the limit $k\longrightarrow\infty$, we need to investigate the eigenvalues of
$\TnK$.

A basis for $\VK$ can be found which is a subset of the (finite) set of all Lie
monomials in $x_1,\ldots,x_n$ with $\leq{}K$ brackets. With respect to this basis, the matrix
of $\TnK$ will be block lower triangular, the blocks being determined by the numbers of
brackets, since under substitution in a monomial with $r$  brackets, all terms will have at
least $r$  brackets. The blocks on the diagonal in this matrix come from that part of the
substitution which retains the same number of brackets, that is
\[
x_i\rightarrow\tfrac1{n-1}(x_1+\ldots+\widehat{x_i}+\ldots+x_n),\quad1\leq{}i\leq{}n\eqno{(2)}
\]
In particular, the (0,0) block will be the action on terms with no brackets, that is on the
span of $x_1,\ldots,x_n$, and will be the matrix induced by (2), namely an $n\times{}n$
matrix all of whose entries are $\frac1{n-1}$ except for zeroes on the diagonal. This matrix
is diagonalizable and its eigenvalues are 1 (simple) and $-\frac1{n-1}$ (with multiplicity
$n-1$). Choose a diagonalizing basis, say $y_1=x_1+\cdots+x_n$ and $y_i=x_{i-1}-x_i$, $i=2,\ldots,n$.

\noindent The $(r,r)$ block of the matrix for $\TnK$ will be given by the action of (2) on
linear combinations of monomials with exactly $r$ brackets. Change basis to a new basis which
are monomials in $y_1,\ldots,y_n$.  The $(r,r)$ block of the matrix for $\TnK$ in this new basis will be given by the action of the substitution,
\[y_1\rightarrow{}y_1,\quad{}y_i\rightarrow-\tfrac1{n-1}y_i\quad(2\leq{}i\leq{}n)\]
Under this substitution any Lie monomial in the $y_i$ with $r$ brackets, will scale by a factor $(-\frac1{n-1})^{r+1-s}$  where $s$ is the number of times that $y_1$ appears in the monomial. Such a monomial with $r$ brackets contains $r+1$ (not necessarily distinct) $y_i$'s; so the exponent here is always non-negative, while it can only vanish if $s=r+1$, that is a monomial containing only $y_1$. Since $[y_1,y_1]=0$, such a monomial can only be a basis element if $r=0$.  In conclusion, the $(r,r)$ block of the matrix for $\TnK$ in this new basis is diagonal with entries which are powers of $(-\frac1{n-1})$, with strictly positive exponents for $r>0$. But a lower triangular block matrix whose diagonal blocks are all diagonal matrices is just a lower triangular matrix. The entries on the diagonal are all non-negative integer powers of $(-\frac1{n-1})$ and the entry 1 (exponent zero) occurs only from $y_1$ in the $(0,0)$ block. We conclude that $\TnK$ has all its eigenvalues which are non-negative integer powers of $(-\frac1{n-1})$, while the eigenvalue 1 appears without multiplicity.

\noindent While the matrix for $\TnK$  may not be diagonalizable, it can be expressed in Jordan block form and the diagonal entries of those blocks must be the eigenvalues, that is there is a single size one Jordan block with eigenvalue 1 and the remaining Jordan blocks all have eigenvalues which are positive powers of $(-\frac1{n-1})$ with varying multiplicities and block sizes. The powers of any Jordan block with eigenvalue $\lambda$, $|\lambda|<1$, converge to zero, and so in a basis with respect to which $\TnK$  is in Jordan normal form (say with the Jordan block of eigenvalue 1 first),  $(\TnK)^k$ converges as $k\rightarrow\infty$  to a matrix identically zero except for the (1,1) position which is 1.  Denote the first basis element by $v_1\in\VK$; it is an eigenvector of $\TnK$ of eigenvalue 1. It depends on $n$, but we omit this dependence for ease of notation. The conclusion is that for any $x\in\VK$,
\[(\TnK)^kx\longrightarrow{}av_1\quad\hbox{as $k\rightarrow\infty$}\]
where $a$ is the first component of $x$ in the new basis (the coefficient of $v_1$).

\noindent Applying this for $x=x_i$, we obtain that the truncation of $x_i^k$ to $\VK$, has a limit  as $k\rightarrow\infty$, and this limit is a multiple of the same eigenvector $v_1$, say $a_iv_1$, where $a_i$ is the first component of $x_i$ with respect to the new basis. In order to find $a_i$, it suffices to truncate to the first block $V^{[0]}$ on which we know that $T_n^{[0]}$ is diagonal with respect to the basis $\{y_i\}$, and the eigenvector with eigenvalue 1 is $y_1$. However,
\[x_1=\tfrac1ny_1+\tfrac1n\sum\limits_{j=2}^n(n+1-j)y_j,\quad{}x_i=x_1-\sum\limits_{j=2}^iy_j\]
and so the coefficient of $y_1$ in any $x_i$ is $\frac1n$. Thus $a_i=\frac1n$ for all $i$ and so the truncations to $\leq{}K$ Lie brackets of the sequences $x_i^k$ have a common limit as $k\rightarrow\infty$, whose zero bracket part is $\frac1ny_1$, for all $k\in\N$. Hence also the sequences $x_i^k$ themselves have a common limit, which we denote by $\mu_n(x_1,\ldots,x_n)$ and whose zero bracket part is $\frac1n(x_1+\cdots+x_n)$.

\noindent Finally, inductively we see that $\mu_n$ is totally symmetric under interchange of the $x_i$'s. This holds for $n=2$ by section 3. Assuming $\mu_{n-1}$ is symmetric (for some $n\geq3$), the symmetry of the construction of the sequences $x_i^k$ means that their limit must also be symmetric. To be more precise, any permutation of $x_1,\ldots,x_n$ will induce an identical permutation of $x_1^k,\ldots,x_n^k$, for each $k$, and so knowing that they share a common limit, it must be invariant under the symmetric group.
\end{proof}

\begin{proof} {\bf of Theorem 4.2:}
Using the construction of $\mu_n$ in Lemma 4.1, we prove inductively that
\[(u_{x_i}(a)=b \hbox{\ for $1\leq{}i\leq{}n$}) \implies u_{\mu_n(x_1,\ldots,x_n)}(a)=b\]
For $n=2$ it is known by \S3. Assume it is true for $n-1$, some $n\geq3$. Suppose that $x_1,\ldots,x_n\in\glie$, $a,b\in{}X$ are such that $u_{x_i}(a)=b$ for all $i=1,\ldots,n$. In the notation of Lemma~4.1, we see by induction on $k$ (starting from $k=0$ which is the initial assumption) that $u_{x_i^k}(a)=b$ for $i=1,\ldots,n$ and all $k\in\N$. Since $u_x(a)$ depends continuously on $x$ and $x_i^k\longrightarrow\mu_n(x_1,\ldots,x_n)$ as $k\rightarrow\infty$, the inductive step follows.
\end{proof}

In the course of the proof of Lemma~4.1, we showed that at every truncation $K$, the operator $\TnK$ has exactly one eigenvalue 1 with all the rest of modulus strictly less than 1, while the iteration converges to a multiple of this eigenvector, that is to a fixed point of $\TnK$. It follows that $\mu_n$ itself is a fixed point of $T_n$, and that this condition determines $\mu_n$ up to scaling.

\begin{lemma}
$\mu_n(x_1,\ldots,x_n)$ is defined uniquely up to scaling, by the property
\begin{align*}
\mu_n\bigg(\mu_{n-1}(\widehat{x_1},\ldots,x_n),&\ldots,\mu_{n-1}(x_1,\ldots,\widehat{x_i},\ldots,x_n),
\ldots,\mu_{n-1}(x_1,\ldots,\widehat{x_n})\bigg)\\
&=\mu_n(x_1,\ldots,x_n)
\end{align*}
\end{lemma}

It is this fact which we use to compute the first few coefficients in an explicit expansion of $\mu_n$ in the next section.

\begin{lemma}
$\mu_n(-x_1,\ldots,-x_n)=-\mu_n(x_1,\ldots,x_n)$
\end{lemma}
\begin{proof}
This follows by induction from Lemma 3.6(i) and Lemma 4.1.
\end{proof}

As a corollary, the expansion of $\mu_n(x_1,\ldots,x_n)$ will involve only odd numbers of
symbols and therefore even numbers of brackets.

\section{Expansion of $\mu_n$ in Lie brackets}

In this section we obtain formula for the expansion of $\mu_n(x_1,\ldots,x_n)$ in Lie brackets, up to the third order. From Lemma 4.3, only even orders are present in $\mu_n$.
The zeroth order is the average $\mu_n^{[0]}=\frac1n(x_1+\cdots+x_n)$ by Lemma 4.1.

\noindent Let $\Vr$ denote the piece of the free Lie algebra $L[x_1,\ldots,x_n]$ spanned by Lie monomials with exactly $r$ brackets.  In order to obtain a spaning set, it is sufficient to enumerate Lie monomials of the form $[x_{i_1},[x_{i_2},\cdots,[x_{i_r},x_{i_{r+1}}]\cdots]]$.

\noindent The symmetric group $S_n$ permutes the $x_i$'s and thus acts on $\Vr$; denote the invariant subspace under this action by $\Vbr$. It is spanned by the images of Lie monomials with exactly $r$ brackets under the symmetrization map $\b$ defined by $\bar{w}\equiv\frac1{n!}\sum\limits_{\sigma\in{}S_n}\sigma(w)$.

\noindent{\bf Second order:} The space $\overline{V}^{[2]}$ is spanned by the symmetrizations of the different types of Lie monomials with two brackets $[x_i,[x_j,x_k]]$, where the type is determined by coincidences or otherwise in the list $i,j,k$. Thus there are {\it prima facie} two generators, coming from symmetrizations of
\[
[x_i,[x_j,x_k]], [x_i,[x_i,x_j]]
\]
where in this list the indices ($i,j,k$) are all considered distinct. However, on interchange of $j$ and $k$ the first expression changes sign and so its symmetrization vanishes. Thus $\overline{V}^{[2]}$ is one-dimensional, generated by
\[
v^{[2]}\equiv\sum\limits_{i,j,i\not=j}[x_i,[x_i,x_j]]=\sum\limits_{i,j,i<j}[x_i-x_j,[x_i,x_j]]
\]
In particular, $\mu_n^{[2]}$ is some multiple of this element, so that
\[
\mu_n^{[\leq2]}=\tfrac1n\sum\limits_ix_i+c_n\sum\limits_{i,j,i\not=j}[x_i,[x_i,x_j]]
\]
for some scalar $c_n$. To find $c_n$, use the defining property of $\mu_n$ from Lemma 4.2. The part with at most two Lie brackets in the left hand side of the equality in Lemma 4.2 is
\begin{align*}
&\tfrac1n\sum\limits_i\Big(\tfrac1{n-1}\sum\limits_{j\not=i}x_j+c_{n-1}\sum\limits_{j\not=i,k\not=i {\rm \ distinct}}[x_j,[x_j,x_k]]\Big)\\
&\quad+c_n\sum\limits_{i,j,i\not=j}\Bigg[\tfrac1{n-1}\sum\limits_{k\not=i}x_k,
   \Big[\tfrac1{n-1}\sum\limits_{l\not=i}x_l,\tfrac1{n-1}\sum\limits_{m\not=j}x_m\Big]\Bigg]\\
&=\tfrac1{n(n-1)}\sum\limits_{i,j\not=i}x_j
+\tfrac{c_{n-1}}n\sum\limits_{i,j,k{\rm \ distinct}}[x_j,[x_j,x_k]]\\
&\quad+\tfrac{c_n}{(n-1)^3}\sum\limits_{i,j,k,l,m,i\not=j,k,l, j\not=m}[x_k,[x_l,x_m]]\\
&=\tfrac1{n(n-1)}\sum\limits_{j}\Big(\sum\limits_{i\not=j}1\Big)x_j
+\tfrac{c_{n-1}}n\sum\limits_{j,k{\rm \ distinct}}\Big(\sum\limits_{i\not=j,k}1\Big)[x_j,[x_j,x_k]]\\
&\quad+\tfrac{c_n}{(n-1)^3}\sum\limits_{k,l,m}\Big(\sum\limits_{i,j{\rm \ distinct},i\not=k,l,j\not=m}1\Big)[x_k,[x_l,x_m]]
\end{align*}
However the set of $i,j$ distinct with $i\not=k,l$ and $j\not=m$ has order
\begin{align*}
(n-3)(n-2)+(n-1)\quad&\hbox{for $k,l,m$ distinct}\\
(n-2)^2+(n-1)\quad&\hbox{for $k=l\not=m$}\\
(n-2)^2\quad&\hbox{for $k=m\not=l$}\\
(n-2)^2\quad&\hbox{for $l=m\not=k$}\\
(n-1)(n-2)\quad&\hbox{for $k=l=m$}\\
\end{align*}
However the sum of $[x_k,[x_l,x_m]]$ over $k,l,m$ with the various conditions in the five cases above, vanishes in the first, fourth and fifth cases (because of antisymmetry of the expression and symmetry of the condition under interchanging $l$ and $m$). In the second and third cases, these sums are,
\[
\sum\limits_{k,m{\rm \ distinct}}[x_k,[x_k,x_m]],\quad
\sum\limits_{k,l{\rm \ distinct}}[x_k,[x_l,x_k]]
\]
respectively, which are identical except for a sign. The conclusion is that
\[
\sum\limits_{k,l,m}\Big(\sum\limits_{i,j{\rm \ distinct},i\not=k,l,j\not=m}1\Big)[x_k,[x_l,x_m]]
=(n-1)\sum\limits_{k,m{\rm \ distinct}}[x_k,[x_k,x_m]]
\]
and thus the $\leq2$ Lie bracket part of the LHS of Lemma 4.2 simplifies to
\[
\tfrac1n\sum\limits_jx_j+\tfrac{n-2}nc_{n-1}\sum\limits_{j,k{\rm \ distinct}}[x_j,[x_j,x_k]]
+\tfrac{c_n}{(n-1)^2}\sum\limits_{k,m{\rm \ distinct}}[x_k,[x_k,x_m]]
\]
Identifying this with $\mu_n^{[\leq2]}$ leaves
\[c_n=\tfrac{n-2}nc_{n-1}+\tfrac{c_n}{(n-1)^2}\]
which simplifies to $n^2c_n=(n-1)^2c_{n-1}$ for $n>2$. This implies that $n^2c_n$ is independent of $n$ and thus shares the value at $n=2$ which is $-\frac1{12}$, since $c_2=-\frac1{48}$ by \S3. Hence $c_n=-\frac1{12n^2}$ and
\[
\fbox{$\mu_n^{[\leq2]}=\tfrac1n\sum\limits_ix_i-\tfrac1{12n^2}\sum\limits_{i,j,i\not=j}[x_i,[x_i,x_j]]$}
\]

\section{Properties of $\mu_n$}

The core properties of $\mu_n$ used to identify its first few coefficients in \S5 are that (i) it is totally symmetric (Lemma 4.1), (ii) it involves only even numbers of Lie brackets (Lemma 4.3), (iii) its zeroth order part is the usual average (Lemma 4.1), and (iv) it is invariant under the substitution map $T_n$ of \S4 (Lemma 4.2). These uniquely determine $\mu_n$. In addition there are two properties which follow inductively from Lemma 3.6(ii),(iii) and the construction of $\mu_n$ as the limit of a sequence.

\begin{lemma}
\begin{itemize}
\item[]
\item[(i)] $\mu_n(\BCH(z,x_1),\ldots,\BCH(z,x_n))=\BCH(z,\mu_n(x_1,\ldots,x_n))$
\item[(ii)] $\mu_n(\BCH(x_1,z),\ldots,\BCH(x_n,z))=\BCH(\mu_n(x_1,\ldots,x_n),z)$
\end{itemize}
\end{lemma}

\noindent{\bf Non-uniqueness of $\mu_n$:} The expressions $\mu_n$ for $n>2$ are not uniquely determined by the universality condition in Theorem 3.2 (without the condition of Lemma 4.2). Indeed,
\[x=\BCH(y,\mu_n(x_1,\ldots,x_n))\]
will also satisfy the conditions, for any $y$ totally symmetric in the $x_i$'s for which $u_y(a)=a$ necessarily follows from $u_{x_i}(a)=b$ for all $i$. If it is additionally known that the set $V_a$ of $y$ satisfying $u_y(a)=a$ is a vector space, then the fact that it is closed under $\BCH$ while $\BCH(x_i,-x_1)\in{}V_a$ ensures that $V_a$ is closed under Lie bracket and the non-uniqueness follows from the following lemma.

\begin{lemma}
The $S_n$-invariant part of the sub-Lie algebra of $L[x_1,\ldots,x_n]$ generated by $\BCH(x_i,-x_1)$, $i=2,\ldots,n$, is non-trivial for $n>2$.
\end{lemma}

A generalization of Lemma 4.2 is the following.

\begin{conjecture} For $1\leq{}m\leq{}n$,
\[\mu_{\binom{n}{m}}\big(\mu_m(x_S)\big|S\subset\{1,\ldots,n\},|S|=m\big)=\mu_n(x_1,\ldots,x_n)
\eqno{(3)}\]
\end{conjecture}

This is trivial for $m=1$ (setting $\mu_1$ to be the identity) and is Lemma 4.2 for $m=n-1$.

\noindent{\bf Comparison up to second order:} Using the results of \S5, up to second order the left hand side of (3) becomes
\begin{align*}
&\tfrac{1}{\binom{n}{m}}\sum\limits_{|S|=m}\Bigg(\tfrac1m\sum\limits_{j\in{}S}x_j
-\tfrac1{12m^2}\sum_{i,j\in{}S,i\not=j}[x_i,[x_i,x_j]]\Bigg)\\
&-\frac1{12{\binom{n}{m}}^2}\sum_{|S|=|T|=m,S\not=T}
\tfrac1{m^3}\Bigg[\sum\limits_{i\in{}S}x_i,
\Big[\sum\limits_{j\in{}S}x_j,\sum\limits_{k\in{}T}x_k\Big]\Bigg]\\
&=\tfrac1{m\binom{n}m}\sum\limits_j\Big(\sum\limits_{S,|S|=m,S\ni{}j}1\Big)x_j
-\tfrac{1}{12m^2\binom{n}m}\sum\limits_{i,j,i\not=j}
\Big(\sum\limits_{S,|S|=m,S\ni{}i,j}1\Big)[x_i,[x_i,x_j]]\\
&-\frac1{12m^3{\binom{n}{m}}^2}\sum_{i,j,k}\Big(\sum_{S,T,|S|=|T|=m,\ S\not=T,\ S\ni{}i,j,\ T\ni{}k}1\Big)
[x_i,[x_j,x_k]]
\end{align*}
The last term breaks up into parts (as in the calculation of \S5) according to the coincidences amongst the indices $i,j,k$. Since $[x_i,[x_j,x_k]]$ is antisymmetric in $j$, $k$, the only non-trivial sums appearing will be when $i=j\not=k$ or $i=k\not=j$, and here these sums will be identical except for sign. In these two cases, the possible distinct sets $S$, $T$ of order $m$ for which $S\ni{}i,j$ and $T\ni{}k$ are divided into cases
\begin{align*}
(S\ni{}i,\ S\not\ni{}k,\ T\ni{}k)\hbox{\ or }(S\ni{}i,k,\ T\ni{}i,k,\ S\not=T)
\hbox{\ or }(S\ni{}i,k,\ T\ni{}k,\ T\not\ni{}i)&\\
(S\ni{}i,j,\ T\ni{}i,j,\ S\not=T)
\hbox{\ or }(S\ni{}i,j,\ T\ni{}i,\ T\not\ni{}j)&
\end{align*}
for $i=j\not=k$ and $i=k\not=j$, respectively.
The difference in these enumerations is thus $\binom{n-2}{m-1}\binom{n-1}{m-1}$ and so
\[\sum_{i,j,k}\ \sum_{S,T,|S|=|T|=m,S\not=T,S\ni{}i,j,T\ni{}k}
[x_i,[x_j,x_k]]=\binom{n-2}{m-1}\binom{n-1}{m-1}\sum_{i,j,i\not=j}[x_i,[x_i,x_j]]\]
so that the LHS of (3) up to second order simplifies to
\begin{align*}
&\tfrac{\binom{n-1}{m-1}}{m\binom{n}m}\sum\limits_jx_j
-\tfrac{\binom{n-2}{m-2}}{12m^2\binom{n}m}\sum\limits_{i,j,i\not=j}[x_i,[x_i,x_j]]
-\tfrac{\binom{n-2}{m-1}\binom{n-1}{m-1}}{12m^3{\binom{n}{m}}^2}\sum_{i,j,i\not=j}[x_i,[x_i,x_j]]\\
&=\tfrac1n\sum\limits_jx_j
-\tfrac{m-1}{12mn(n-1)}\sum\limits_{i,j,i\not=j}[x_i,[x_i,x_j]]
-\tfrac{n-m}{12mn^2(n-1)}\sum_{i,j,i\not=j}[x_i,[x_i,x_j]]\\
&=\tfrac1n\sum\limits_jx_j
-\tfrac{1}{12n^2}\sum\limits_{i,j,i\not=j}[x_i,[x_i,x_j]]
\end{align*}
which coincides with $\mu_n^{[\leq2]}$. Thus the conjecture holds up to second (or third) order in Lie brackets.

\section{An example: $SL(2,\R)$}

Consider the action of $G=SL(2,\R)$ on $X=\H$, the upper-half plane, defined by
\[\left(\begin{array}{cc}a&b\\ c&d\end{array}\right)\cdot{}z=\frac{az+b}{cz+d}\]
The corresponding infinitesimal action is of $\glie$ by the vector field
\[\left(\begin{array}{rr}a&b\\ c&-a\end{array}\right)\cdot{}z
\equiv\frac{d}{dt}\Bigg(\left(\begin{array}{cc}1+ta&tb\\ tc&1-ta\end{array}\right)\cdot{}z\Bigg)\Bigg|_{t=0}=-cz^2+2az+b\>,\]
the corresponding flow on $\H$ being given by $\dot{z}=-cz^2+2az+b$.
This has one fixed point in $\H$ whenever $a^2+bc<0$. Conversely, for each point $z\in\H$, there is a unique element of $\glie$ up to scaling which fixes it. This action is simple enough that the flow can be easily explicitly solved. From this, it can be determined that those elements of $\glie$ for which the flow in unit time fixes a point $z\in\H$ can be described as a union of two sets,  a one-dimensional subspace along with elements for which $a^2+bc$ is a negative integer multiple of $\pi^2$, for the latter the action of flow by unit time is the identity. For example, those elements which flow $i$ to itself in unit time (but for which the flow by unit time is not the identity) are precisely $\left(\begin{array}{cc}0&b\\ -b&0\end{array}\right)$, $b\in\R$.

\noindent Given two points $z,w\in\H$, there is a one-parameter family of elements of $\glie$ which flows $z$ to $w$ in unit time. For example, those elements of $G$ which take $-1+i$ to $1+i$ are
\[
\left(\begin{array}{cc}\cos{t}+\sin{t}&2\cos{t}\\ \sin{t}&\cos{t}+\sin{t}\end{array}\right), t\in\R
\]
By directly solving the flow equation, one finds that the elements of $\glie$ which flow $-1+i$ to $1+i$ in unit time are all of the form
\[
\left(\begin{array}{cc}0&b\\ c&0\end{array}\right),
\]
where $(b,c)=(2,0), (0,1)$ or for $bc>0$ it is parametrised by ($bc=t^2$)
\[b=t\big(\coth{t}+\sqrt{\csch^2{t}-1}\ \big),\ {}c=\frac12t\big(\coth{t}-\sqrt{\csch^2{t}-1}\ \big)\ (0<\sinh|t|<1)\]
while for $bc<0$ it is parametrised by ($bc=-t^2$)
\[b=t\big(\cot{t}+\sqrt{(\csc{t})^2+1}\ \big),\quad{}c=\frac12t\big(\cot{t}-\sqrt{(\csc{t})^2+1}\ \big),\quad t\not\in\pi\Z\]
In either parametrisation, as $t\rightarrow0+$, $(b,c)\rightarrow(2,0)$, while as
$t\rightarrow0-$, $(b,c)\rightarrow(0,1)$. Denote the subset of $\glie$ so defined by $W$. A plot of the associated points $(b,2c)$ in the plane $\R^2$ is shown below.

\[\includegraphics[width=.6\textwidth]{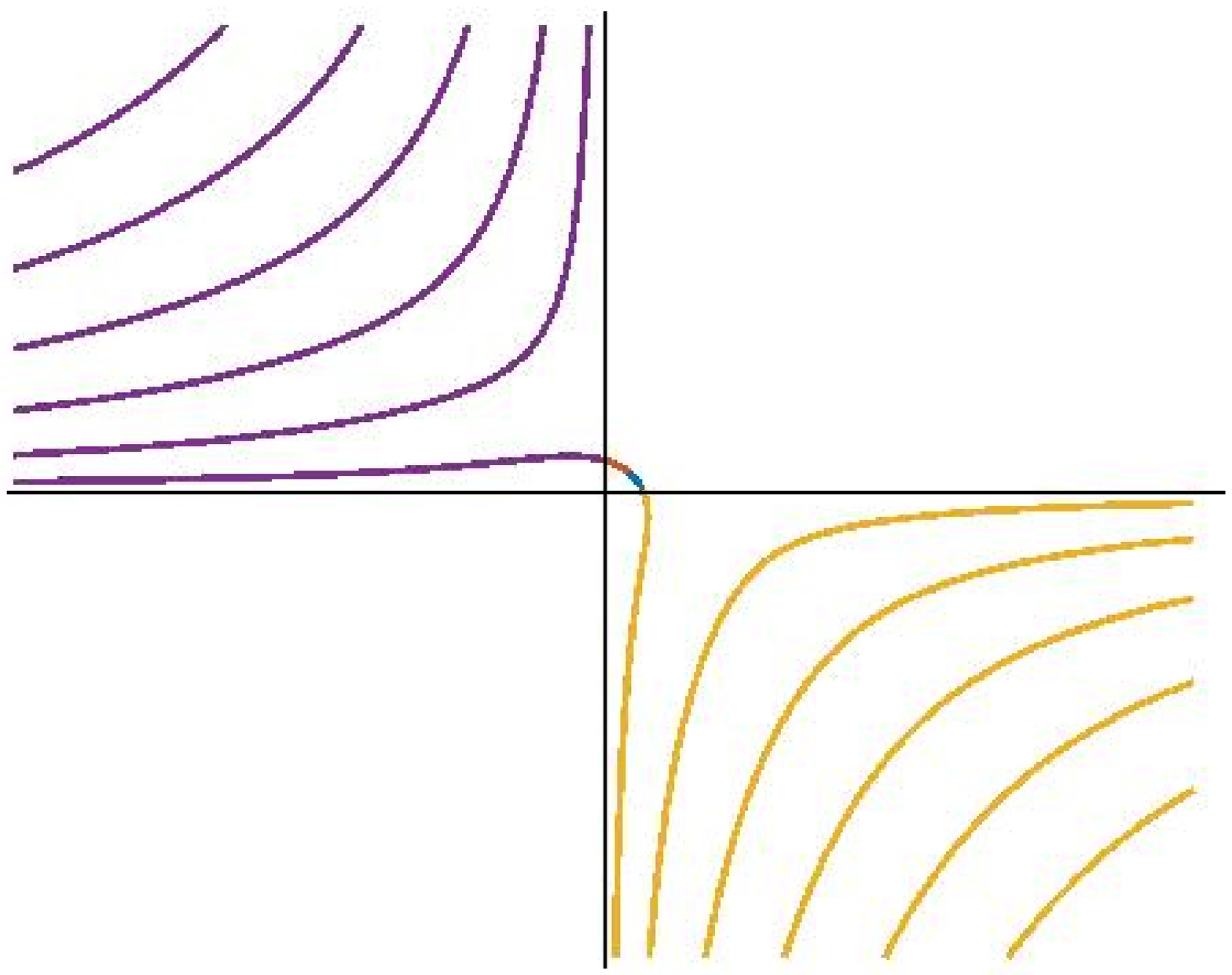}\]

\noindent The connection between the infinitesimal action and the action of $G$ are by the formulae
\begin{align*}
\exp{\left(\begin{array}{cc}0&b\\ c&0\end{array}\right)}
&=\cosh\sqrt{bc}\cdot{}I+\frac{\sinh\sqrt{bc}}{\sqrt{bc}}\left(\begin{array}{cc}0&b\\ c&0\end{array}\right)\quad\hbox{for $bc>0$}\\
\exp{\left(\begin{array}{cc}0&b\\ c&0\end{array}\right)}
&=\cos\sqrt{-bc}\cdot{}I+\frac{\sin\sqrt{-bc}}{\sqrt{-bc}}\left(\begin{array}{cc}0&b\\ c&0\end{array}\right)\quad\hbox{for $bc<0$}
\end{align*}

The result of this paper defines maps $\mu_n:W_i^n\longrightarrow{}W_i$ for each component $W_i$ of $W$.

Since $w_0\equiv\left(\begin{array}{cc}0&2\\ 0&0\end{array}\right)\in{}W$ and $\left(\begin{array}{cc}1&2\\ -1&-1\end{array}\right)\in\glie$ fixes $-1+i$, hence
\[
\BCH\left(\left(\begin{array}{cc}0&2\\ 0&0\end{array}\right),
s\left(\begin{array}{cc}1&2\\ -1&-1\end{array}\right)\right)\in{}W
\]
for $s\in\R$. In this parametrisation, according to Lemma 6.1(i), $\mu_n$ will act as the arithmetic mean on the parameters $s$.

\section{Conclusions and applications}

In \S5, we only managed to determine $\mu_n$ up to third order in brackets.
For specific values of $n$, it is possible to use Lemma 4.2, to computationally determine the sequences of coefficients in $\mu_n$, using a Hall basis for the free Lie algebra. This was carried out up to third order in \cite{Si}, but in principle the same technique could be used to higher orders. Knowing the existence of recurrence relations with polynomial coefficients for these sequences, formulae can then be derived for them as functions of $n$. It is unclear which method is faster to determine higher order coefficients, a direct analytic technique or using computations for specific $n$. Note that the dimension of the totally symmetric part of the four bracket part of the free Lie algebra is greater than 1, so that unlike in \S5, the recurrence relations obtained even for the 4-bracket part will be systems of homogeneous linear recurrence relations with polynomial coefficients.

As noted in \S3, the motivating example for this paper is that of DGLA models of cell complexes. For a cubical cell, pick two opposite vertices and consider the six minimal paths from one to the other. By the functorial nature of the dependence of $A(X)$ on $X$, a model of the cube can be obtained from that of a six-faceted banana; see \cite{GL2}. Starting from a symmetric model of the six-faceted banana as generated from $\mu_6$, the generated model of the cube will share those symmetries of the cube which preserve the chosen diagonal.

It is expected that there may be other applications of universal averages, particularly as they seem to satisfy various interesting relations as in \S6.

\bigskip\noindent{\bf Acknowledgements:} This research was supported in part by Grant No 2016219 from the United States-Israel Binational Science Foundation (BSF). The authors wish to thank D. Kazhdan for the suggestion to look at the example of $SL_2(\R)$.


\begin{thebibliography}{00}

\bibitem{B4} U.~Buijs, Y.~F\'elix, A.~Murillo, D.~Tanr\'e, Maurer-Cartan elements in the Lie models of finite simplicial complexes, {\tt Canad.\ Math.\ Bull.} {\bf 60} (2017), 470--477
 {\tt arXiv:1606.08794 [math.AT]}

\bibitem{D} E.~Dynkin, Calculation of the coefficients in the Campbell–Hausdorff formula,
     {\it Dokl. Akad. Nauk USSR} (in Russian) {\bf57} (1947), 323-326

 \bibitem{E} M.~Eichler, A new proof of the Baker-Campbell-Hausdorff formula,
     {\it J. Math. Soc. Japan} {\bf 20} (1968), 23-25

 \bibitem{GL} I.~Griniasty, R.~Lawrence, An explicit symmetric DGLA model of a triangle, {\it Higher Structures}, {\bf 3}(1) (2019), 1-16

 \bibitem{GL2} I.~Griniasty, R.~Lawrence, Explicit DGLA models of higher dimensional cells, {\it Preprint} (2019)

 \bibitem{GGL} N.~Gadish, I.~Griniasty, R.~Lawrence, An explicit symmetric DGLA model
     of a bi-gon, {\tt arXiv:1705.08483},  {\it J.  Knot Th. Ramif.} (2019)

 \bibitem{LS} R.~Lawrence,  D.~Sullivan, A formula for topology/deformations and its
    significance, {\it Fundamenta Mathematicae} {\bf 225} (2014) 229-242


\bibitem{Q} D.~Quillen, Rational homotopy theory, {\it Ann. of Math.} (2) {\bf 90} (1969), 205--295

\bibitem{Si} M.~Siboni, Universal averages in gauge actions, {\sl Minor thesis}, Hebrew University (2019)

\bibitem{S} D.~Sullivan, Infinitesimal computations in topology, {\it Inst. Hautes \'Etudes Sci. Publ. Math.} {\bf 47} (1977), 269--331

\bibitem{TZ} T.~Tradler, M.~Zeinalian, Infinity structure of Poincar\'e duality spaces, {\it Algebr.~Geom.~Topol.} {\bf 7} (2007), 233--260, {\tt arXiv:math/0309455 [math.AT]}


\end{thebibliography}
\end{document}